\documentclass[11pt,a4paper]{amsart} 

\usepackage{fullpage}

\usepackage{amssymb}
\usepackage{amsmath,amsthm}
\usepackage{verbatim}
\usepackage[colorlinks = true, citecolor = blue]{hyperref} 

\newcommand{\dimension}{d}

\newcommand{\degreetwo}{2}

\newcommand{\C}{\mathbb{C}}
\newcommand{\R}{\mathbb{R}}
\newcommand{\euclideanspace}{\R^{\dimension}}
\newcommand{\Z}{\mathbb{Z} }
\newcommand{\lattice}{\Z^{\dimension}}
\newcommand{\T}{\mathbb{T} }
\newcommand{\torus}{\mathbb{T}^\dimension}
\newcommand{\N}{\mathbb{N} }
\newcommand{\Zmod}[1]{\Z/#1\Z}
\newcommand{\unitsmod}[1]{(\Z/#1 \Z)^\times}

\newcommand{\latticepoint}{m}
\newcommand{\toruspoint}{\xi}
\newcommand{\spacepoint}{x}
\newcommand{\freqpoint}{\xi} 

\newcommand{\contFT}[1]{\widetilde{#1}}
\newcommand{\arithmeticFT}[1]{\widehat{#1}}
\newcommand{\torusFT}[1]{\widehat{#1}}
\newcommand{\latticeFT}[1]{\widehat{#1}}

\newcommand{\spheremeasure}{\sigma}
\newcommand{\arithmeticsphere}{S^{\dimension-1}(\radius)} 
\newcommand{\arithmeticspheremeasure}[1]{\sigma_{#1}}

\newcommand{\radius}{r}

\newcommand{\radii}{\mathcal{R}}
\newcommand{\acceptableradii}{\mathcal{R}_{full}}

\newcommand{\sequencegrowth}{w}

\newcommand{\norm}[1]{\left\Arrowvert #1 \right\Arrowvert}

\newcommand{\absolutevalueof}[1]{\lvert #1 \rvert}

\newcommand{\operatornorm}[2]{\left\Arrowvert #2 \right\Arrowvert_{#1}}
\newcommand{\lpnorm}[2]{\left\| #2 \right\|_{\ell^{#1}(\Z^\dimension)}}
\newcommand{\ellpoperatornorm}[1]{\operatornorm{p \to p}{#1}}
\newcommand{\inparentheses}[1]{\left( #1 \right)}
\newcommand{\inbraces}[1]{\left\{ #1 \right\}}
\newcommand{\inbrackets}[1]{\left[ #1 \right]}
\newcommand{\setof}[1]{\left\{ #1 \right\}}

\newcommand{\gcdof}[2]{\left( #1, #2 \right)}
\newcommand{\eof}[1]{e \left( #1 \right)}

\newcommand{\inverse}[1]{\frac{1}{#1}}
\newcommand{\convolvedwith}{*}

\newcommand{\fxn}{f}
\newcommand{\bumpfxn}{\varphi}
\newcommand{\charfxn}[1]{{\bf 1}_{#1}}

\newcommand{\largenum}{N}

\newcommand{\primepower}{k}
\newcommand{\orderofprime}[1]{v_\oddprime(#1)}
\newcommand{\totientfunction}[1]{\varphi(#1)}

\newcommand{\modulus}{q}
\newcommand{\unit}{a}

\newcommand{\twistedGausssum}[1]{G(\unit,\modulus,#1)}

\newcommand{\twistedKloostermansum}[1]{K(\modulus,\largenum,#1)}

\newcommand{\singularseries}[1]{\mathfrak{S}(#1)}

\newcommand{\oddprime}{\mathfrak{p}}
\newcommand{\theprimes}{\mathfrak{P}}
\newcommand{\goodprimes}{\mathfrak{P}_{good}}
\newcommand{\badprimes}{\mathfrak{P}_{bad}}
\newcommand{\goodmodulus}{\modulus_{good}}
\newcommand{\badmodulus}{\modulus_{bad}}

\newcommand{\dilateby}[1]{\operatorname{Dil}_{#1}}
\newcommand{\Kloostermankernel}{K^{\modulus}_\radius}
\newcommand{\Kloostermanoperator}{C^{\modulus}_\radius}
\newcommand{\error}{{E}_{\radius}}
\newcommand{\HLop}{C_{\radius}}

\newcommand{\avgop}{A_{\radius}}
\newcommand{\maxop}{A_*}
\newcommand{\sequentialarithmeticsphericalmaxfxn}{A_{\radii}^* \fxn}
\newcommand{\sequentialarithmeticsphericalmaxop}{A_{\radii}^* }
\newcommand{\lacunaryavgfxn}[1]{A_{#1} \fxn}
\newcommand{\arithmeticsphericalaveragefxn}{A_{\radius}\fxn} 
\newcommand{\arithmeticsphericalavgop}[1]{A_{#1}} 
\newcommand{\arithmeticsphericalmaxfxn}{A_*\fxn}
\newcommand{\arithmeticsphericalmaxop}{A_*}

\newtheorem{prop}{Proposition}[section]
\newtheorem{lemma}{Lemma}[section]
\newtheorem{thm}{Theorem}

\newtheorem{cor}{Corollary}[section]

\newtheorem{conjecture}{Conjecture}

\newtheorem*{lacunarythm}{Calder\'on, Coifman--Weiss}
\newtheorem*{Gauss}{The Gauss bound}
\newtheorem*{Weil}{The Weil bound for Kloosterman sums}
\newtheorem*{MSW_transference}{Magyar--Stein--Wainger transference lemma}
\newtheorem*{approximationlemma}{The Approximation Formula}

\theoremstyle{remark}
\newtheorem{rem}{Remark}[section]
\newtheorem{question}{Question}[section]

\title{The discrete spherical averages over a family of sparse sequences}
\author{Kevin Hughes}
\email{Kevin.Hughes@ed.ac.uk}
\address{
School of Mathematics 
\\ 
The University of Edinburgh 
\\ 
James Clerk Maxwell Building 
\\ 
The King's Buildings 
\\ 
Peter Guthrie Tait Road 
\\ 
EDINBURGH 
\\ 
EH9 3FD 
Scotland, UK
}

\begin{document}
\maketitle
\tableofcontents

\begin{abstract}
We initiate the study of the $\ell^p(\Z^\dimension)$-boundedness of the arithmetic spherical maximal function over sparse sequences. 
We state a folklore conjecture for lacunary sequences, a key example of Zienkiewicz and prove new bounds for a family of sparse sequences that achieves the endpoint of the Magyar--Stein--Wainger theorem for the full discrete spherical maximal function in \cite{MSW_spherical}. 
Perhaps our most interesting result is the boundedness of a discrete spherical maximal function in \(\Z^4\) over an infinite, albeit sparse, set of radii. 
Our methods include the Kloosterman refinement for the Fourier transform of the spherical measure (introduced in \cite{Magyar_discrepancy}) and Weil bounds for Kloosterman sums which are utilized by a new further decomposition of spherical measure. 
\end{abstract}

%
\section{Introduction}
%
\subsection{Stein's spherical maximal function and its arithmetic analogue}
%
In \cite{Stein_spherical}, Stein introduced the spherical maximal function and proved that it is bounded on $L^p(\R^\dimension)$ for $p>\frac{\dimension}{\dimension-1}$ and $\dimension \geq 3$. This was later extended to $p>2$ when $\dimension=2$ by Bourgain in \cite{Bourgain_circular}. 
Recently, discrete analogues of Stein's spherical maximal function have been considered. 
The discrete sphere of radius $\radius \geq 0$ in $\Z^\dimension$ is 
$\arithmeticsphere := \{x\in \Z^\dimension : |x|^2=r^2 \}$ 
which contains $N_{\dimension}(\radius) = \# \arithmeticsphere$ lattice points. 
For dimensions $\dimension \geq 4$, the set $\arithmeticsphere$ is non-empty precisely when $\radius^2 \in \N$. 
Let $\acceptableradii$ denote the set of radii $\radius$ such that $\arithmeticsphere \not= \emptyset$, then $\acceptableradii$ is precisely $\setof{\radius \in \R_{\geq 0} : \radius^2 \in \N}$ when $\dimension \geq 4$. 
For $\radius \in \acceptableradii$, we introduce the discrete spherical averages: 
\begin{equation}
\arithmeticsphericalaveragefxn(x) 
= \inverse{N_{\dimension}(\radius)} \sum_{y \in \arithmeticsphere} \fxn(x-y) 
= \fxn \convolvedwith \arithmeticspheremeasure{\radius} (x)
\end{equation}
where $\arithmeticspheremeasure{\radius} := \frac{1}{N_\dimension(\radius)} {\bf 1}_{ \{x\in \Z^\dimension : |x|^\degreetwo=r^\degreetwo\} }$ is the uniform probability measure on $\arithmeticsphere$. 
The associated (full) maximal function is 
\begin{equation}
\arithmeticsphericalmaxfxn 
= \sup_{\radius \in \acceptableradii} \absolutevalueof{\arithmeticsphericalaveragefxn} 
. 
\end{equation}
Motivated by Stein's theorem, it is natural to ask: \emph{when is $\maxop$ bounded on $\ell^p(\Z^\dimension)$?} 
Testing the maximal operator on the delta function and using the asymptotics for the number of lattice points on spheres, $N_\dimension(\radius) \eqsim \radius^{\dimension-2}$ when $\dimension \geq 5$, we expect that the maximal operator is bounded on $\ell^p$ for $p>\frac{\dimension}{\dimension-2}$ when $\dimension \geq 5$. In fact, building on the work of \cite{Magyar_dyadic}, this was proven in \cite{MSW_spherical} with a subsequent restricted weak-type bound at the endpoint $p=\frac{\dimension}{\dimension-2}$ proven in \cite{Ionescu_spherical}. 
In particular, $\maxop$ is a bounded operator from $\ell^{p,1}(\Z^\dimension)$ to restricted $\ell^{p,\infty}(\Z^\dimension)$ for $p=\frac{\dimension}{\dimension-2}$; that is, $\maxop$ is restricted weak-type $(\frac{\dimension}{\dimension-2},\frac{\dimension}{\dimension-2})$. 
This result is sharp. 
For generalizations to higher degree varieties where the sharp ranges of \(\ell^{p,\infty}(\Z^\dimension)\) are unknown, we refer the reader to \cite{Magyar_ergodic} and \cite{Hughes_Vinogradov}.

\subsection{The lacunary spherical maximal function and its arithmetic analogue}
Shortly after Stein's work on the spherical maximal function \cite{Stein_spherical}, it was observed by Calder\'on and Coifman--Weiss that lacunary versions of Stein's spherical maximal function are bounded on a larger range of $L^p(\R^\dimension)$-spaces than for the full Stein spherical maximal function -- see \cite{Calderón_lacunary_spherical} and \cite{Coifman_Weiss_lacunary_spherical} respectively. In particular, they proved: 
\begin{lacunarythm}
The lacunary (continuous) spherical maximal function is bounded on $L^p(\R^\dimension)$ for $\dimension \geq 2$ and $1 < p \leq \infty$. 
\end{lacunarythm}

Similarly, we define the lacunary discrete spherical maximal function when $\dimension \geq 5$ by restricting the set of radii to lie in a lacunary sequence $\radii := \{ \radius_{j} \}_{j \in \N} \subset \acceptableradii$.
Recall that a sequence is lacunary if $\radius_{j+1} > c \, \radius_j$ for some $c > 1$. 
More generally, for any \(\radii \subseteq \acceptableradii\), the \emph{discrete spherical maximal function over $\radii$} is defined in the natural way as 
\begin{equation}
\sequentialarithmeticsphericalmaxfxn 
:= \sup_{\radius_j \in \radii} \; \absolutevalueof{\lacunaryavgfxn{\radius_j}} 
.
\end{equation}
By the Magyar--Stein--Wainger discrete spherical maximal theorem in \cite{MSW_spherical}, we know that any discrete spherical maximal function in 5 or more dimensions over a subsequence of radii in $\acceptableradii$ is bounded on $\ell^p(\Z^\dimension)$ for $\dimension \geq 5$ and $p > \frac{\dimension}{\dimension-2}$. In particular this holds true for any lacunary subsequence in 5 or more dimensions. 

It is conjectured that the continuous lacunary spherical maximal function is bounded from $L^1(\R^\dimension)$ to $L^{1,\infty}(\R^\dimension)$ for $\dimension \geq 2$. 
See \cite{STW_endpoint_spherical} and \cite{STW_pointwise_lacunary} for recent work in this direction. 
Analogously, it is a folklore conjecture that the arithmetic lacunary spherical maximal function is bounded on $\ell^p(\Z^\dimension)$ for $p>1$.
The following conjecture is our motivation for this paper. 
\begin{conjecture}\label{conjecture:lacunary}
For $\dimension \geq 5$, if $\radii$ is a lacunary subsequence of $\acceptableradii$, then $\sequentialarithmeticsphericalmaxop: \ell^1(\Z^\dimension) \to \ell^{1,\infty}(\Z^\dimension)$.
\end{conjecture}
%

\subsection{It's a trap!}\label{subsection:Ackbar}
Surprisingly, J. Zienkiewicz has shown that Conjecture~\ref{conjecture:lacunary} is false in general.\footnote{This counterexample was communicated to the author by Zienkiewicz after an initial draft of this paper was completed.} 
More precisely, Zienkiewicz proved that there exist infinite, yet arbitrarily thin subsets \(\radii \subset \acceptableradii\) such that \(\sequentialarithmeticsphericalmaxfxn\) is unbounded on \(\ell^p(\Z^\dimension)\) for \(1 \leq p < \frac{\dimension}{\dimension-1}\) and \(\dimension \geq 5\). 
Zienkiewicz's counterexamples proceed by a probabilistic argument that incorporates information about the discrete spherical averages when one reduces mod \(Q\) for \(Q \in \N\). 
By a probabilistic argument, Zienkiewicz constructs counterexamples that violate \eqref{good_primes_estimate} of {\bf G} below for infinitely many primes. 
In Section~\ref{section:conclusion} we revise Conjecture~\ref{conjecture:lacunary} to account for these counterexamples.

\subsection{Results of this paper}
Our main theorem is the following improvement to the range of boundedess for maximal functions over lacunary sequences of radii possessing an interesting dichotomy. 
\begin{thm}\label{thm:variants_of_Euclid}
Let \(\radii := \setof{\radius_j} \subset \acceptableradii\) be a lacunary subsequence of \(\R^+\). 
Assume that \(\radii\) decomposes the primes \(\theprimes \subset \N\) into two (not necessarily disjoint) sets: 
the good primes \(\goodprimes\) and the bad primes \(\badprimes\) such that 
\begin{itemize}
\item[{\bf G}:]\label{item:good_primes_estimate} 
for each $\oddprime \in \goodprimes$, 
\begin{equation}\label{good_primes_are_units}
\setof{\radius_j^2 \mod \oddprime} 
\subset \unitsmod{\oddprime} , 
\end{equation}
and for all $\epsilon>0$, 
\begin{equation}\label{good_primes_estimate}
\# \setof{\radius_j^2 \mod \oddprime} 
\lesssim_{\epsilon} \oddprime^\epsilon 
\end{equation}
where the implicit constants may depend on \(\epsilon\), but not on \(p \in \goodprimes\), 
\item[{\bf B}:]\label{item:bad_primes_estimate} 
and the bad primes satisfy 
\begin{equation}\label{bad_primes_estimate}
\sum_{\oddprime \in \badprimes} \oddprime^{-s} 
< 
\infty \text{ for some $s \in (0,1]$. } 
\end{equation}
\end{itemize}
If \( p \geq \frac{\dimension}{\dimension-(1+s)}\) and \(p > \frac{\dimension-1}{\dimension-2} \), then \(\sequentialarithmeticsphericalmaxop\) is a bounded operator on \(\ell^p(\Z^\dimension)\) for \(\dimension \geq 5\). 

If additionally \(2 \in \goodprimes\), then \(\sequentialarithmeticsphericalmaxop\) is bounded on \(\ell^p(\Z^\dimension)\) for the same range of \(p\) and \(\dimension \geq 4\). 
\end{thm}

Theorem~\ref{thm:variants_of_Euclid} reduces our problem to finding sequences of natural numbers satisfying certain arithmetic properties, and it would be superfluous if we could not find a sequence of radii satisfying the {\bf G} and {\bf B} dichotomy. 
Our next theorem gives a family of sequences satisfying these conditions. 
This family is well known in number theory as it includes \emph{primorials}, also known as \emph{Euclidean primes}, whose definition is motivated by Euclid's proof of the infinitude of primes. 
For these sequences, \eqref{good_primes_estimate} is simple to verify. 
However, \eqref{bad_primes_estimate} is difficult to verify, and we only have very poor bound for it in this article. 
In turn, for our family of sparse sequences, presently we are only able to show that the associated discrete spherical maximal function is strong-type at the Magyar--Stein--Wainger endpoint for the full discrete spherical maximal function as opposed to restricted weak-type bound in \cite{Ionescu_spherical}. 
%
\begin{thm}\label{thm:ENFANT}
Let $\sequencegrowth>1$. 
%
For any fixed \(m \in \N\), the sequence of radii $\radii = \setof{\radius_j \in \R^+ : \radius_j^2 = m + \prod_{j_0 \leq i \leq 2^{j^\sequencegrowth}} \oddprime_i}$ satisfy \eqref{good_primes_are_units} and \eqref{good_primes_estimate} of {\bf G} for all primes and \eqref{bad_primes_estimate} of {\bf B} for \(s=1\). 
\end{thm}
Theorem~\ref{thm:variants_of_Euclid} and Theorem~\ref{thm:ENFANT} immediately combine to yield the following endpoint Magyar--Stein--Wainger theorem applied to such sequences. 
\begin{cor}\label{cor:Euclid}
Let \(\dimension \geq 4\), \(\sequencegrowth > 1\) and \(\radii = \setof{\radius_j \in \R^+ : \radius_j^2 = 1 + \prod_{i \leq 2^{j^\sequencegrowth}} \oddprime_i}\) where \(\oddprime_i\) is the \(i^{th}\) prime. 
Let \(\sequentialarithmeticsphericalmaxop\) denote the spherical maximal function associated to $\radii$. Then \(\sequentialarithmeticsphericalmaxop\) is a bounded operator on \(\ell^p(\Z^\dimension)\) for \(p \geq \frac{\dimension}{\dimension-2}\) and \(\dimension \geq 4\). 
\end{cor}

\begin{rem}
By the Prime Number Theorem, $\prod_{i \leq T} \oddprime_i \eqsim e^{T}$ as $T \to \infty$. 
We see that our sequence grows much faster than lacunary since \(\prod_{i \leq 2^{j^\sequencegrowth}} \oddprime_i \eqsim e^{2^{j^\sequencegrowth}}\) for any $\sequencegrowth > 0$ as $j \to \infty$. 
The existence of thicker sequences with property \eqref{good_primes_estimate} would be interesting. 
On the other hand, our main difficulty in this paper is to establish \eqref{bad_primes_estimate}. 
We only succeed in doing so for $s=1$; hence the limitation to \(p \geq \frac{\dimension}{\dimension-2}\) in Corollary~\ref{cor:Euclid}. 
\end{rem}

An intriguing aspect of Corollary~\ref{cor:Euclid} is that \(\sequentialarithmeticsphericalmaxop\) is bounded on $\ell^2(\Z^4)$. 
This is surprising since the full discrete spherical maximal function, $\arithmeticsphericalmaxop$ fails to be bounded on $\ell^2(\Z^4)$. 
Worse yet, for dimensions $\dimension \leq 4$, the full maximal function is only bounded on $\ell^\infty(\Z^\dimension)$. 
Theorem~\ref{thm:variants_of_Euclid} and Corollary~\ref{cor:Euclid} mark the first results in 4 dimensions for boundedness of the arithmetic spherical maximal function over infinite sequences. 

Let us examine the four dimensional situation further. 
In $\Z^4$, there are precisely 24 lattice points on a sphere of radius $2^j$ for all $j \in \N$, e.g. $N_{4}(2^j) = 24$. 
Applying the discrete spherical maximal function to the delta function demonstrates that the naive definition of our maximal function in 4 dimensions is wrong. However, further considerations suggest that there could be a version of the Magyar--Stein--Wainger theorem in 4 dimensions. 
To make this precise, we must account for some arithmetic phenomena. 
From the work of Hardy--Littlewood on the circle method, we have the asymptotic formula 
\begin{equation}\label{eq:HL_asymptotic}
N_\dimension(\radius) 
= \singularseries{\radius^2} \frac{\pi^{\dimension/2}}{\Gamma(\dimension/2)} \radius^{\dimension-2} +O_{\epsilon}(\radius^{\dimension/2+\epsilon})
\end{equation}
where $\singularseries{\radius^2}$ is the singular series, which satisfies $\singularseries{\radius^2} \eqsim 1$ when $\dimension \geq 5$. 
Lagrange's theorem and Jacobi's four square theorem demonstrate that the 4 dimensional case (i.e. \(S^3(\radius) \in \Z^4\)) is different. 
In four dimensions, the bound for the error term in \eqref{eq:HL_asymptotic} dominates the main term, and therefore, \eqref{eq:HL_asymptotic} is not useful as an asymptotic. 
However, Kloosterman was able to refine their method by exploiting oscillation between Gauss sums to improve \eqref{eq:HL_asymptotic} to 
\begin{equation}\label{equation:Kloosterman}
N_\dimension(\radius) 
= \singularseries{\radius^2} \frac{\pi^{\dimension/2}}{\Gamma(\dimension/2)} \radius^{\dimension-2} +O_{\epsilon}(\radius^{\frac{\dimension}{2}-\frac{1}{9}+\epsilon})
\end{equation}
for all $\epsilon > 0$ and $\dimension \geq 4$. 
The cost here is that the singular series in the asymptotic formula is not uniform, and in fact can be very small. One can predict this from Jacobi's theorem since there are precisely 24 lattice points when $\radius = 2^j$ for all $j \in \N$; in this case, one sees that $\singularseries{4^j} \lesssim 4^{-j}$. 
To avoid this 2-adic obstruction in the singular series when $\dimension = 4$, we make the additional assumption that \emph{$\radius_j^2 \not\equiv 0 (\mod 4)$ for each $j \in \N$} or \eqref{good_primes_are_units} holds for the prime 2. In either case $N_4(\radius^2) \eqsim \radius^2$ so that there are many lattice points on $S^3(\radius)$. Modifying the discrete spherical maximal function in 4 dimensions in this way, it is natural to conjecture that it is bounded on $\ell^p(\Z^4)$ for $2 < p \leq \infty$ -- see \cite{Hughes_thesis} for a precise statement of this conjecture and a related result.

\subsection{Notations}\label{subsection:notations}
%
Our notation is a mix of notations from analytic number theory and harmonic analysis. Most of our notation is standard, but there are a few choices based on aesthetics. 

\begin{itemize}
\item The torus $\T^\dimension$ may be identified with any box in $\euclideanspace$ of sidelengths 1, for instance $[0,1]^\dimension$ or $[-1/2,1/2]^\dimension$.
%
%
\item We identify $\Zmod{\modulus}$ with the set $\inbraces{1, \dots, \modulus}$ and $\unitsmod{\modulus}$ is the group of units in $\Zmod{\modulus}$, also considered as a subset of $\inbraces{1, \dots, \modulus}$. 
\item $\eof{t}$ will denote the character $e^{2 \pi i t}$ for $t \in \R, \Zmod{\modulus}$ or $\T$.
\item We abuse notation by writing ${b}^2$ to mean $\sum_{i=1}^{\dimension} b_i^2$ for $b \in (\Zmod{\modulus})^{\dimension}$ and the dot product notation $b \cdot m$ to mean $\sum_{i=1}^{\dimension} b_i m_i$ for $b, m \in (\Zmod{\modulus})^\dimension$ or \(\Z^\dimension\). 
\item For any $\modulus \in \N$, $\totientfunction{\modulus}$ will denote Euler's totient function, the size of $\unitsmod{\modulus}$. 
\item For two functions $f, g$, $f \lesssim g$ if $\absolutevalueof{f(x)} \leq C \absolutevalueof{g(x)}$ for some constant $C>0$. 
$f$ and $g$ are comparable $f \eqsim g$ if $f \lesssim g$ and $g \lesssim f$. 
All constants throughout the paper may depend on dimension $\dimension$. 
\item 
If $f: \R^\dimension \to \C$, then we define its Fourier transform by $\contFT{\fxn}(\xi) := \int_{\R^\dimension} f(x) e(x \cdot \xi) dx$ for $\xi \in \R^\dimension$. 
If $f: \T^\dimension \to \C$, then we define its Fourier transform by $\torusFT{\fxn}(\latticepoint) := \int_{\T^\dimension} f(x) e(-\latticepoint \cdot x) dx$ for $\latticepoint \in \lattice$. 
If $f:\lattice \to \C$, then we define its inverse Fourier transform by $\latticeFT{\fxn}(\xi) := \sum_{\latticepoint \in \lattice} f(\latticepoint) e(n \cdot \xi)$ for $\xi \in \T^\dimension$. 
\item $\ellpoperatornorm{T}$ will denote the $\ell^p(\Z^\dimension)$ to $\ell^p(\Z^\dimension)$ operator norm of the operator $T$. 
\end{itemize}

\subsection{Layout of the paper}
By interpolation with the usual $\ell^\infty(\Z^\dimension)$ bound, we restrict our attention to the range $1 \leq p \leq 2$. 
From \cite{MSW_spherical}, we understand that each average decomposes into a main term (resembling the singular series and singular integral of the circle method) and an error term. We recall this machinery in section~\ref{section:MSW_machinery}. 
The main term  and error term will be bounded on ranges of $\ell^p(\Z^\dimension)$-spaces by distinct arguments. 
In section~\ref{section:main_term}, our bounds for the main term exploit the Weil bounds for Kloosterman sums via the transference principle of Magyar--Stein--Wainger. 
The main result here is Lemma~\ref{lemma:main_term} and we introduce a more precise decomposition of the multipliers in order to use the Kloosterman method. 
In section~\ref{section:error_term}, the error term is handled by a square function argument using the lacunary condition. The main lemma here is Lemma~\ref{lemma:error_term_lemma}. Our novelty here is that we exploit (well-known) cancellation for averages of Ramanujan sums to improve the straight-forward $\ell^1(\Z^\dimension)$ bound. 
Theorem~\ref{thm:variants_of_Euclid} follows immediately by combining Lemma~\ref{lemma:main_term} with Lemma~\ref{lemma:error_term_lemma}. 
In section~\ref{section:ENFANT_example}, we prove Theorem~\ref{thm:ENFANT}. 
The properties of our sequences are well known to analytic number theorists, but we could not find them in the literature. 
Section~\ref{section:conclusion} concludes our paper with some questions and remarks.

%
\section{MSW machinery and the Kloosterman refinement}\label{section:MSW_machinery}
Before turning to the proof of Theorem~\ref{thm:variants_of_Euclid}, we review the Kloosterman refinement as in (1.9) of Lemma~1 from \cite{Magyar_discrepancy}, some machinery from \cite{MSW_spherical} and bounds for exponential sums. 
%

Let $\arithmeticspheremeasure{\radius} := \frac{1}{N_\dimension(\radius)} \charfxn{\setof{\latticepoint \in \lattice : \absolutevalueof{\latticepoint} = \radius}}$ denote the normalized surface measure on the sphere of radius $\radius$ centered at the origin for some $\radius \in \radii$. The circle method of Hardy--Littlewood, and of Kloosterman yields $N_\dimension(\radius) \eqsim \radius^{\dimension-2}$ for $\radius \in \acceptableradii$ when $\dimension \geq 5$ and for $\radius^2 \not\equiv 0 \mod 4$ when $\dimension=4$, so we renormalize our spherical measure to 
\begin{equation}
\arithmeticspheremeasure{\radius} 
:= \radius^{2-\dimension} \cdot \charfxn{\setof{\latticepoint \in \lattice : \absolutevalueof{\latticepoint} = \radius}} 
. 
\end{equation}
Note that our subsequences of radii $\radii$ exclude the case $\radius^2 \equiv 0 \mod 4$ when $\dimension=4$, so that we may renormalize in this case when 2 is a good prime; that is, 2 satisfies \eqref{good_primes_are_units} of {\bf G }. 
Furthermore, we renormalize our averages and maximal function accordingly. 
Using Heath-Brown's version of the Kloosterman refinement to the Hardy--Littlewood--Ramanujan circle method from \cite{HB_cubic_forms}, Magyar gave an approximation formula generalizing \eqref{equation:Kloosterman} for $\arithmeticspheremeasure{\radius}$ in \cite{Magyar_discrepancy}. 
We recall this now:
\begin{approximationlemma}
If $\dimension \geq 4$, then for each $\radius \in \acceptableradii$, 
\begin{equation}\label{eq:Kloosterman_approximation}
\arithmeticFT{\arithmeticspheremeasure{\radius}}(\toruspoint) 
= 
\sum_{\modulus=1}^{\radius} \sum_{\latticepoint \in \lattice} K(\modulus,\radius^2;\latticepoint) \Psi(\modulus \toruspoint - \latticepoint) \contFT{d\spheremeasure_\radius}(\toruspoint - \latticepoint/\modulus) + \arithmeticFT{\error}(\toruspoint)
\end{equation}
with error term, $\error$ that is the convolution operator given by the multiplier $\arithmeticFT{\error}$, satisfying 
\begin{equation}\label{eq:Kloosterman_error_bound}
\norm{\error \fxn}_{\ell^2(\lattice)} \lesssim_\epsilon \radius^{2-\frac{\dimension+1}{2}+\epsilon} \norm{\fxn}_{\ell^2(\lattice)} 
\end{equation}
for any $\epsilon > 0$. 
\end{approximationlemma}
Here and throughout, for $\modulus, N \in \N$ and $\latticepoint \in \Z^\dimension$, 
\begin{equation}
K(\modulus, N; \latticepoint) 
:= \modulus^{-\dimension} \sum_{\unit \in \unitsmod{\modulus}} \eof{ -\frac{\unit N}{\modulus} } \sum_{b \in \inparentheses{\Zmod{\modulus}}^\dimension} \eof{\frac{\unit b^\degreetwo + b \cdot \latticepoint}{\modulus}} 
\end{equation}
are \emph{Kloosterman sums}, $\Psi$ is a smooth function supported in $[-1/4, 1/4]^\dimension$ and equal to 1 on $[-1/8, 1/8]^\dimension$. 
Our Kloosterman sums arise naturally in Waring's problem as a weighted sum of the Gauss sums 
\begin{equation}
\twistedGausssum{\latticepoint} 
:= \modulus^{-\dimension} \sum_{b \in \inparentheses{\Zmod{\modulus}}^\dimension} \eof{\frac{\unit b^\degreetwo + b \cdot \latticepoint}{\modulus}} 
\end{equation}
($\modulus \in \N$, $\unit \in \unitsmod{\modulus}$ and $\latticepoint \in \Z^\dimension$) 
so that $K(\modulus, N; \latticepoint) = \sum_{\unit \in \unitsmod{\modulus}} \eof{ -\frac{\unit N}{\modulus} } \twistedGausssum{\latticepoint}$. 
%
\(d\spheremeasure_\radius\) denotes the induced Lebesgue measure on the sphere of radius \(\radius\) in \(\R^\dimension\) normalized so that the total surface measure is $\pi^{\dimension/2} / \Gamma(\dimension/2)$ for each $\radius > 0$. 
Note that this spherical measure is also the restriction of the Gelfand--Leray form to the sphere of radius \(\radius\), or the Dirac delta measure; both with the appropriate normalization.
One may take $\toruspoint=0$ to check that \eqref{eq:Kloosterman_approximation} is compatible with \eqref{eq:HL_asymptotic} (keep in mind our renormalization). 
\begin{rem}
The bound for the error term in \eqref{eq:Kloosterman_error_bound} was obtained with a weaker exponent of $2-\frac{\dimension}{2}-\frac{1}{9}+\epsilon$ in place of $2-\frac{\dimension}{2}-\frac{1}{2}+\epsilon$ for the dyadic maximal function version in \cite{Hughes_thesis} by extending Kloosterman's original method in \cite{Kloosterman} while Magyar achieved the (presumably optimal) savings of \eqref{eq:Kloosterman_error_bound} using Heath-Brown's method in \cite{HB_cubic_forms}. Alternately, Heath-Brown's method in \cite{HB_quadratic_forms} achieves \eqref{eq:Kloosterman_error_bound}. 
\end{rem}
%

With The Approximation Formula in mind, it is necessary to understand the relationship between multipliers defined on $\torus$ and $\euclideanspace$. Suppose that $\mu$ is a multiplier supported in $[-1/2,1/2]^\dimension$, then we can think of $\mu$ as a multiplier on $\euclideanspace$ or $\torus$; denote these as $\mu_{\euclideanspace}$ and $\mu_{\torus}$ respectively where $\mu_{\torus}(\toruspoint):= \sum_{\latticepoint \in \lattice} \mu(\toruspoint-\latticepoint)$ is the periodization of $\mu_{\euclideanspace}$. These have convolution operators $T_{\euclideanspace}$ and $T_{\torus}$ on their respective spaces. 
Explicitly, for $F:\euclideanspace \to \C$,
\[
T_{\euclideanspace}F(\spacepoint) := \int_{\euclideanspace} \mu_{\euclideanspace}(\freqpoint) \contFT{F}(\freqpoint) \eof{-\spacepoint \cdot \freqpoint} \, d\freqpoint
\]
and for $f:\lattice \to \C$,
\[
T_{\lattice} f(\latticepoint) := \int_{\torus} \mu_{\torus}(\toruspoint) \arithmeticFT{f}(\toruspoint) \eof{-\latticepoint \cdot \toruspoint} \, d\toruspoint .
\]
%
%
We will need apply these to maximal functions, so we extend these notions to Banach spaces. 
Let $B_1, B_2$ be two finite dimensional Banach spaces with norms $\norm{\cdot}_1, \norm{\cdot}_2$, and $\mathcal{L}(B_1,B_2)$ is the space of bounded linear tranformations from $B_1$ to $B_2$. Let $\ell^p_{B_i}$ be the space of functions $f:\lattice \to B_i$ such that $\sum_{\latticepoint \in \lattice} \norm{f}_{i}^p < \infty$ and $L^p_{B_i}$ be the space of functions $F:\euclideanspace \to B_i$ such that $\int_{\euclideanspace} \norm{F}_{i}^p < \infty$. 
For a fixed modulus $\modulus \in \N$, suppose that $\mu: [-1/2\modulus,1/2\modulus]^\dimension \to \mathcal{L}(B_1,B_2)$ is a multiplier with convolution operators $T_{\euclideanspace}$ on $\euclideanspace$ and $T_{\lattice}$ on $\lattice$. Extend $\mu$ periodically to the torus to define 
$\mu^{\modulus}_{\torus}(\freqpoint) := \sum_{\latticepoint \in \lattice} \mu_{\torus}(\freqpoint-\latticepoint/\modulus)$ 
with convolution operator $T^{\modulus}_{\lattice}$ on $\lattice$ defined by $\arithmeticFT{T^{\modulus}_{\lattice}f}(\freqpoint) = \mu^{\modulus}_{\torus}(\freqpoint) \cdot \arithmeticFT{f}(\freqpoint)$. 
Magyar--Stein--Wainger proved a transference principle which relates the boundedness of $T_{\R^\dimension}$ to that of $T_{\Z^\dimension}^{\modulus}$ for any finite dimensional Banach space. 
The following transference principle is Proposition~2.1 in \cite{MSW_spherical}: 
\begin{MSW_transference} 
For $1 \leq p \leq \infty$,
\begin{equation}\label{tranference_principle}
\norm{T^{\modulus}_{\lattice}}_{\ell^p_{B_1} \to \ell^p_{B_2}} 
\lesssim \norm{T_{\euclideanspace}}_{L^p_{B_1} \to L^p_{B_2}} .
\end{equation}
The implicit constant is independent of $B_1,B_2,p$ and $\modulus$.
\end{MSW_transference}

We will apply this lemma with $B_1 = B_2 = \ell^\infty(\N)$ in order to compare averages over the discrete spherical maximal function with known bounds for averages over the continuous lacunary spherical maximal function. Technically, we should truncate the maximal function and apply the lemma with $B_1 = B_2 = \ell^\infty(\setof{1, \dots, N})$ for arbitrarily large $N \in \N$ with bounds independent of \(N\). However, this is a standard technique that we will not emphasize. 


The Magyar--Stein--Wainger transference lemma allows us to utilize our understanding of the continuous theory for spherical averages and reduces our problem to understanding the arithmetic aspects of the multipliers $\sum_{\latticepoint \in \lattice} K(\modulus,\radius^2;\latticepoint) \Psi(\modulus \toruspoint - \latticepoint) \contFT{d\spheremeasure}(\radius(\toruspoint - \latticepoint/\modulus))$ for each $\modulus$. To handle these we recall Proposition~2.2 in \cite{MSW_spherical}: 

\begin{lemma}[Magyar--Stein--Wainger]\label{lemma:MSW_periodic_inequality}
Suppose that $\mu(\toruspoint) = \sum_{\latticepoint \in \lattice} g({\latticepoint}) \bumpfxn(\toruspoint - \latticepoint/\modulus)$ is a multiplier on $\torus$ where $\bumpfxn$ is smooth and supported in $[-1/2\modulus,1/2\modulus]^\dimension$ with convolution operator $T$ on $\lattice$. Furthermore, assume that $g({\latticepoint})$ is $\modulus$-periodic ($g(\latticepoint_1) = g(\latticepoint_2)$ if $\latticepoint_1 \equiv \latticepoint_2 \mod{\modulus}$). For a $q$-periodic sequence, define the \((\Zmod{\modulus})^\dimension\)-Fourier transform $\arithmeticFT{g}(\latticepoint) := \sum_{b \in \inparentheses{\Zmod{\modulus}}^\dimension} g(b) e\left( \frac{\latticepoint \cdot b}{\modulus} \right)$. Then for $1\leq p \leq 2$,
\begin{equation}\label{eq:MSW_periodic_inequality}
\norm{T}_{\ell^p(\lattice) \to \ell^p(\lattice)}
\lesssim \left( \sup_{m \in \inparentheses{\Zmod{\modulus}}^\dimension} \lvert g({m}) \rvert \right)^{2-2/p} \left( \sup_{n \in \inparentheses{\Zmod{\modulus}}^\dimension} \lvert \arithmeticFT{g}({n}) \rvert \right)^{2/p-1} 
\end{equation}
with implicit constants depending on $\bumpfxn$ and $p$, but independent of $g$. 
\end{lemma}

We will apply Lemma~\ref{lemma:MSW_periodic_inequality} with the sequence $g(\latticepoint)$ taken to be the Kloosterman sums $K(\modulus, \radius^2;\latticepoint)$. 
We have the following estimates for the Kloosterman and Gauss sums. 
\begin{Gauss}[(12.5) on p. 151 of \cite{Grosswald}]
For all $\latticepoint \in \lattice$, 
\begin{equation}
\absolutevalueof{\twistedGausssum{\latticepoint}} 
\leq 2^{\dimension/2}\modulus^{-\dimension/2} 
. 
\end{equation}
\end{Gauss}
Applying the triangle inequality, we immediately obtain the {Gauss bound} for Kloosterman sums:  
\begin{equation}
\absolutevalueof{K(\modulus,\latticepoint; N)} 
\leq 2^{\dimension/2} \modulus^{1-\dimension/2} 
. 
\end{equation}
Kloosterman beat the Gauss bound for Kloosterman sums by making use of oscillation between Gauss sums in the Kloosterman sums, and consequently, he extended the Hardy--Littlewood circle method for representations of diagonal quadratic forms in 5 or more variables down to 4 variables. 
Similarly here, the Gauss bound is insufficient for our purposes and we need to make use of cancellation between the Gauss sums. The first type of bound to appear for this is due to Kloosterman in \cite{Kloosterman}. The best possible estimate of this sort is \emph{Weil's bound} which essentially obtains square-root cancellation in the average over $\unit \in \unitsmod{\modulus}$. 
\begin{Weil}[(1.13) of \cite{Magyar_discrepancy}]
For each modulus $\modulus \in \N$ write $\modulus = \modulus_{odd} \cdot \modulus_{even}$ where $\modulus_{odd}$ is odd while $\modulus_{even}$ is the precise power of 2 that divides $\modulus$. 
For all $\epsilon>0$, we have 
\begin{equation}\label{eq:Weil_bound}
\absolutevalueof{\twistedKloostermansum{\latticepoint}} 
\lesssim_{\epsilon} \modulus^{-\frac{\dimension-1}{2}+\epsilon} \gcdof{\modulus_{odd}}{\largenum}^{1/2} \modulus_{even}^{1/2} 
\end{equation}
where the implicit constants are independent of $\modulus$ and uniform in $\latticepoint \in \Z^\dimension$. 
\end{Weil}

\begin{rem}
We note that in our definition of the Kloosterman sums $K(\modulus, N, \latticepoint)$, we have the following important multiplicativity property: if $\gcdof{\modulus_1}{\modulus_2}=1$, then for any $N \in \N$ and $\latticepoint \in \Z^\dimension$, 
\begin{equation}\label{eq:Kloosterman_multiplicativity}
K(\modulus_1 \modulus_2, N; \latticepoint) 
= K(\modulus_1, N; \latticepoint) K(\modulus_2, N; \latticepoint) 
. 
\end{equation}
For a proof see Lemma 5.1 in \cite{Davenport}. 
\end{rem}

%
\section{The main term}\label{section:main_term}
%
Our starting point is The Approximation Formula; we have 
\[
\latticeFT{\arithmeticspheremeasure{\radius}}(\toruspoint) 
= \sum_{\modulus=1}^{\radius} \latticeFT{\HLop^{\modulus}}(\toruspoint)  + \arithmeticFT{\error}(\toruspoint)
\]
where $\latticeFT{\HLop^{\modulus}}(\toruspoint)$ is the multiplier
\[
\sum_{\latticepoint \in \lattice} K(\modulus,\radius^2;\latticepoint) \Psi(\modulus \toruspoint - \latticepoint) \contFT{d\spheremeasure_\radius}(\toruspoint - \latticepoint/\modulus) 
. 
\]
Let $\HLop^{\modulus}$ be the convolution operator with multiplier $\latticeFT{\HLop^{\modulus}}$. Then letting $\HLop := \sum_{1 \leq \modulus \leq \radius} \HLop^{\modulus}$, we have $\arithmeticsphericalavgop{\radius} = \HLop + \error$ for each $\radius \in \acceptableradii$. 
The main goal of this section is to prove the following lemma regarding the main terms $\HLop$. 
We will discuss the error terms $\error$ in Section~\ref{section:error_term}. 
\begin{lemma}\label{lemma:main_term}
If \(\radii \subset \acceptableradii\) is a lacunary subsequence of radii satisfying \eqref{good_primes_are_units}, \eqref{good_primes_estimate} of {\bf G} and \eqref{bad_primes_estimate} of {\bf B} for some \(s \in [0,1]\), 
then for $\dimension \geq 5$, 
\begin{equation}\label{eq:main_term_bound}
\lpnorm{p}{\sup_{\radius \in \radii} \absolutevalueof{\HLop \fxn}} 
\lesssim 
\lpnorm{p}{\fxn} 
\end{equation}
if 
\(
\frac{\dimension}{\dimension-(1+s)} 
\leq  
p 
\leq 
2 
\)
and simultaneously 
\(
\frac{\dimension-1}{\dimension-2} 
< 
p 
\leq 
2 . 
\)
Furthermore, if $\dimension = 4$ and 2 is a good prime (\(2 \in \goodprimes\)), then \eqref{eq:main_term_bound} is true for the same range of \(p\). 
\end{lemma}
%
Before proving Lemma~\ref{lemma:main_term} we orient ourselves with a few propositions. 
All implicit constants are allowed to depend on the dimension $\dimension$ and $p$. 
To start, we have the triangle inequality for any subsequence $\radii \subseteq \acceptableradii$, 
\begin{equation}\label{eq:singular_series_bound}
\ellpoperatornorm{\sup_{\radius \in \radii} \absolutevalueof{\HLop}} 
\leq \sum_{\modulus=1}^\infty \ellpoperatornorm{\sup_{\radius \in \radii}  \absolutevalueof{\HLop^{\modulus}}} 
. 
\end{equation}
We restrict our attention to an individual summand for the time being. 
We have the following bound from \cite{MSW_spherical}.
\begin{prop}[Proposition~3.1 (a) in \cite{MSW_spherical}]\label{prop:Stein+Gauss_bound_for_multipliers}
If $\frac{\dimension}{\dimension-1} < p \leq 2$, then 
\begin{equation*}
\ellpoperatornorm{\sup_{\radius \in \acceptableradii} \absolutevalueof{\HLop^{\modulus}}} 
\lesssim 
\modulus^{1-\frac{\dimension}{p'}} 
. 
\end{equation*}
\end{prop}
\noindent
This bound applies to the full sequence of radii and hence any subsequence, which we will choose to be $\radii$ in a moment. We briefly record that the range of $\ell^p(\Z^\dimension)$-spaces improves if one replaces Stein's theorem (for the spherical maximal function) with the Calder\'on, Coifman--Weiss theorem for any lacunary subsequence of $\acceptableradii$ in the proof of Proposition~\ref{prop:Stein+Gauss_bound_for_multipliers}. 
See Proposition~3.1 (a) in \cite{MSW_spherical} for more details. 
\begin{prop}\label{prop:Gauss_bound_for_multipliers}
If $\radii$ is a lacunary subsequence of $\acceptableradii$ and $1 < p \leq 2$, then 
\begin{equation}\label{eq:CCW+Gauss_bound_for_multipliers}
\ellpoperatornorm{\sup_{\radius \in \radii} \absolutevalueof{\HLop^{\modulus}}} 
\lesssim 
\; \modulus^{1-\frac{\dimension}{p'}} 
. 
\end{equation}
\end{prop}
%
%

In \cite{MSW_spherical} we learned that we can factor $\HLop^{\modulus} = S^{\modulus}_\radius \circ T^{\modulus}_\radius = T^{\modulus}_\radius \circ S^{\modulus}_\radius$ into two commuting multipliers $S^{\modulus}_\radius$ and $T^{\modulus}_\radius$, effectively separating the arithmetic and analytic aspects of $\HLop^{\modulus}$, by using a smooth function $\Psi'$ such that  $\charfxn{\setof{[-1/4,1/4]^\dimension}} \leq \Psi' \leq \charfxn{\setof{[-1/2,1/2]^\dimension}}$ on $\T^\dimension$ so that $\Psi \cdot \Psi' = \Psi$. For $\radius \in \acceptableradii$ and $\modulus \in \N$, we have the \emph{Kloosterman multipliers} 
\begin{equation}
\latticeFT{S^{\modulus}_\radius}(\xi) 
:= \sum_{\latticepoint \in \lattice} K(\modulus,\radius^2;\latticepoint) \Psi(\modulus \toruspoint - \latticepoint) 
\end{equation}
and the localized spherical averaging multipliers 
\begin{equation}
\latticeFT{T^{\modulus}_\radius}(\xi) 
:= \sum_{\latticepoint \in \lattice} \Psi'(\modulus \toruspoint - \latticepoint) \contFT{d\spheremeasure_{\radius}}(\toruspoint - \latticepoint/\modulus) 
. 
\end{equation}
%
In order to improve on the Magyar--Stein--Wainger range of $\ell^p(\Z^\dimension)$-spaces for $\sup_{\radius \in \radii} \absolutevalueof{\HLop}$, 
\footnote{The sharp range of $\ell^p(\Z^\dimension)$-spaces is $p > \frac{\dimension}{\dimension-2}$ when $\radii$ is $\acceptableradii$, which results from summing \eqref{eq:CCW+Gauss_bound_for_multipliers} over $\modulus \in \N$ in \eqref{eq:singular_series_bound}} 
we need to beat the exponent ${1-\frac{\dimension}{p'}}$ of the modulus $\modulus$ in \eqref{eq:CCW+Gauss_bound_for_multipliers}. 
Using the Weil bound, we do so for an individual convolution operator $\HLop^\modulus$. 
\begin{prop}[Weil bound for Kloosterman multipliers]\label{prop:Weil_bound_for_a_multiplier}
If $1 \leq p \leq 2$ and $\modulus$ is an odd number, then for each $\radius \in \acceptableradii$ and for all $\epsilon>0$, 
\begin{equation}\label{eq:Weil_bound_for_a_multiplier}
\ellpoperatornorm{\HLop^{\modulus}} 
\lesssim_\epsilon \totientfunction{\modulus}^{\frac{2}{p}-1} \cdot \modulus^{-\frac{\dimension-1}{p'} + \epsilon} \gcdof{\modulus}{\radius^\degreetwo}^{\frac{1}{p'}} 
. 
\end{equation}
\end{prop}
\begin{proof}
On $\ell^2(\Z^\dimension)$, we apply the Weil bound \eqref{eq:Weil_bound} to the Kloosterman sums in $\latticeFT{S^\modulus_\radius}$. Meanwhile on $\ell^1(\Z^\dimension)$, if $K^{\modulus}_\radius$ denotes the kernel of the multiplier $\latticeFT{S^\modulus_\radius}$, then $K^{\modulus}_\radius(b) = \sum_{\unit \in \unitsmod{\modulus}} \eof{\frac{\unit [\radius^2-b^2]}{\modulus}}$ for $b \in \inparentheses{\Zmod{\modulus}}^\dimension$ where we have the trivial bound of $\totientfunction{\modulus}$. 
\eqref{eq:MSW_periodic_inequality} of Lemma~\ref{lemma:MSW_periodic_inequality} yields the bound. 
\end{proof}

However, for a fixed modulus $\modulus$ in $\N$ with $\modulus = \modulus_1 \modulus_2$ such that $\modulus_1$ and $\modulus_2$ coprime, we can factor $S_\radius^\modulus$ into two pieces. 
If $\gcdof{\modulus_1}{\modulus_2}=1$, then by the Chinese Remainder Theorem and multiplicativity of Kloosterman sums \eqref{eq:Kloosterman_multiplicativity} we have 
\begin{equation}\label{eq:decomposition_of_HLop}
\HLop^{\modulus} 
= T^{\modulus}_\radius 
  \circ U^{1, \modulus}_\radius 
  \circ U^{2, \modulus}_\radius 
\end{equation}
where the operators $U^{1, \modulus}_\radius$ and $U^{2, \modulus}_\radius$ are defined by the multipliers 
\begin{align*}
& \latticeFT{U^{1, \modulus}_\radius}(\toruspoint) 
:= \sum_{\latticepoint \in \lattice} K(\modulus_1,\radius^2;\latticepoint) \Psi_{\modulus_1 \modulus_2}'~\inparentheses{\toruspoint - \frac{\latticepoint}{\modulus_1 \modulus_2}} 
\\ 
& \latticeFT{U^{2, \modulus}_\radius}(\toruspoint) 
:= \sum_{\latticepoint \in \lattice} K(\modulus_2,\radius^2;\latticepoint) \Psi_{\modulus_1 \modulus_2}~\inparentheses{\toruspoint - \frac{\latticepoint}{\modulus_1 \modulus_2}} 
\end{align*}
since 
\begin{align*}
\latticeFT{S^{\modulus}_\radius}(\xi) 
& = \sum_{\latticepoint \in \lattice} K(\modulus_1 \modulus_2,\radius^2;\latticepoint) \Psi_{\modulus_1 \modulus_2}(\toruspoint - \frac{\latticepoint}{\modulus_1 \modulus_2}) 
\\ & = \sum_{\latticepoint \in \lattice} K(\modulus_1,\radius^2;\latticepoint) K(\modulus_2,\radius^2;\latticepoint) \Psi_{\modulus_1 \modulus_2}~\inparentheses{\toruspoint - \frac{\latticepoint}{\modulus_1 \modulus_2}} \Psi_{\modulus_1 \modulus_2}'~\inparentheses{\toruspoint - \frac{\latticepoint}{\modulus_1 \modulus_2}} 
\\ & = \inparentheses{\sum_{\latticepoint \in \lattice} K(\modulus_1,\radius^2;\latticepoint) \Psi_{\modulus_1 \modulus_2}'~\inparentheses{\toruspoint - \frac{\latticepoint}{\modulus_1 \modulus_2}} } 
   \inparentheses{\sum_{\latticepoint \in \lattice} K(\modulus_2,\radius^2;\latticepoint) \Psi_{\modulus_1 \modulus_2}~\inparentheses{\toruspoint - \frac{\latticepoint}{\modulus_1 \modulus_2}} } 
. 
\end{align*}
Note that $U^{1, \modulus}_{\radius}$ is $\modulus_1$-periodic in $\radius^\degreetwo$ and $U^{2, \modulus}_{\radius}$ is $\modulus_2$-periodic in $\radius^\degreetwo$ while both of $K(\modulus_1,\radius^2;\latticepoint)$ and $K(\modulus_2,\radius^2;\latticepoint)$ are $\modulus_1 \modulus_2$-periodic in $\latticepoint \in \lattice$. 

Using our refined decomposition \eqref{eq:decomposition_of_HLop}, we now come to the main proposition that enables us to prove Lemma~\ref{lemma:main_term}. 
\begin{prop}\label{prop:counting_bound_for_maximal_multipliers}
Fix $\modulus \in \N$ such that $\modulus = \modulus_1 \modulus_2$ with $\gcdof{\modulus_1}{\modulus_2}=1$ and $\radii$ a lacunary subsequence of $\acceptableradii$. 
Let $\radii_{i (\modulus_1)}$ denote the set of radii $\setof{\radius \in \radii : \radius^\degreetwo \equiv i \mod \modulus_1}$. 
If $1 < p \leq 2$, then 
\begin{equation}\label{eq:counting_bound_for_maximal_multipliers}
\ellpoperatornorm{\sup_{\radius \in \radii} |\HLop^{\modulus}|} 
\lesssim \modulus_2^{1-\dimension/p'} 
  \cdot \# \setof{i \in \Zmod{\modulus_1} : \radii_{i (\modulus_1)} \not= \emptyset} 
  \cdot \sup_{i \in \Zmod{\modulus_1}} \inbraces{\ellpoperatornorm{U^{1, \modulus}_{\radius_i}}} 
\end{equation} 
where $\radius_i$ is a chosen representative of $\radii_{i(\modulus_1)}$ for each $i \in \Zmod{\modulus_1}$. 
\end{prop}
It will be important in our proof of Lemma~\ref{lemma:main_term} that $\# \setof{i \in \Zmod{\modulus} : \radii_{i (\modulus)} \not= \emptyset}$ is small for \emph{most} moduli $\modulus$ and that we can apply Proposition~\ref{prop:Weil_bound_for_a_multiplier}, the Weil bound for Kloosterman multipliers to the operators $U^{1, \modulus}_{\radius_i}$.
\begin{proof}[Proof of Proposition~\ref{prop:counting_bound_for_maximal_multipliers}]
Let $\modulus = \modulus_1 \modulus_2$ and subset $\radii \subset \acceptableradii$ be a lacunary subsequence. The union bound applied to $\radii = \cup_{i \in \Zmod{\modulus_1}} \radii_{i (\modulus_1)}$ implies 
\begin{equation}\label{eq:good_prime_projected_bound}
\ellpoperatornorm{\sup_{\radius \in \radii} \absolutevalueof{\HLop^{\modulus}}} 
\leq \sum_{i=1}^{\modulus_1} \ellpoperatornorm{\sup_{\radius \in \radii_{i (\modulus_1)}} \absolutevalueof{\HLop^{\modulus}}} 
, 
\end{equation}
with the understanding that if $\radii_{i (\modulus_1)}$ is empty, then $\ellpoperatornorm{\sup_{\radius \in \radii_{i (\modulus_1)}}{\absolutevalueof{\HLop^{\modulus}}}}$ is 0. 
Therefore, \eqref{eq:counting_bound_for_maximal_multipliers} will follow from proving 
\begin{equation}
\ellpoperatornorm{\sup_{\radius \in \radii_{i (\modulus_1)}} \absolutevalueof{\HLop^{\modulus}}} 
\lesssim 
\modulus_2^{1-\dimension/p'} \ellpoperatornorm{ {U^{1, \modulus}_{\radius_i}}} 
. 
\end{equation}

Our decomposition \eqref{eq:decomposition_of_HLop} implies that 
\[
\sup_{\radius \in \radii_{i (\modulus_1)}} \absolutevalueof{\HLop^{\modulus}\fxn} 
= \sup_{\radius \in \radii_{i (\modulus_1)}} \absolutevalueof{T^\modulus_\radius S^\modulus_{\radius} \fxn} 
= \sup_{\radius \in \radii_{i (\modulus_1)}} \absolutevalueof{T^\modulus_\radius U^{2,\modulus}_{\radius} U^{1,\modulus}_{\radius} \fxn} 
. 
\]
If $\radius_1, \radius_2 \in \radii_{i (\modulus_1)}$, then $U^{1,\modulus}_{\radius_1} = U^{1,\modulus}_{\radius_2}$. 
Therefore, if $\radius_i$ is a chosen representative radius in $\radii_{i (\modulus_1)}$, then 
\[
\sup_{\radius \in \radii_{i (\modulus_1)}} \absolutevalueof{\HLop^{\modulus}\fxn} 
= \sup_{\radius \in \radii_{i (\modulus_1)}} \absolutevalueof{ {T^\modulus_\radius U^{2,\modulus}_{\radius}} \inparentheses{U^{1,\modulus}_{\radius_i} \fxn} } 
. 
\]
The operator $T^\modulus_\radius U^{2,\modulus}_{\radius}$ is very similar to $\HLop^{\modulus_2}$, and in fact \eqref{eq:CCW+Gauss_bound_for_multipliers} holds with $\HLop^{\modulus_2}$ replaced by $T^\modulus_\radius U^{2,\modulus}_{\radius}$ since $U^{2,\modulus}_{\radius}$ is $\modulus_2$-periodic in $\radius^2$ and $K(\modulus_2,\radius^2;\latticepoint)$ are $\modulus_1 \modulus_2$-periodic in $\latticepoint \in \lattice$.
(Likewise, $U^{1,\modulus}_{\radius_i}$ is very similar to $\HLop^{\modulus_1}$ and the Weil bound \eqref{eq:Weil_bound_for_a_multiplier} applies with $\HLop^{\modulus_1}$ replaced by $U^{1,\modulus}_{\radius_i}$.) 
The Magyar--Stein--Wainger transference principle combined with the Calderon, Coifman--Weiss theorem and \eqref{eq:CCW+Gauss_bound_for_multipliers} imply \eqref{eq:counting_bound_for_maximal_multipliers}
since $\radii_{i (\modulus_1)}$ is also a (possibly finite) lacunary sequence. 
\end{proof}
%

%
\subsection{Proof of Lemma~\ref{lemma:main_term}} 
Recall that $\theprimes$ denotes the set of primes in $\N$. 
In this section we fix our collection of radii to be a lacunary sequence \(\radii \subset \acceptableradii\) so that the set of primes \(\theprimes = \badprimes \cup \goodprimes \) is a union of the sets bad primes \(\badprimes\) and good primes \(\goodprimes\) satisfying \eqref{bad_primes_estimate} and \eqref{good_primes_are_units}, \eqref{good_primes_estimate} respectively. 
%
%
%

If $\oddprime$ is a good prime, then lifting \eqref{good_primes_estimate} to $\Zmod{\oddprime^\primepower}$ for $\primepower \in \N$ implies 
\[
\# \setof{i \in \Zmod{\oddprime^\primepower} : \radii_{i (\oddprime^\primepower)} \not= \emptyset} 
\lesssim_{\epsilon} \oddprime^{\primepower-1+\epsilon}
\]
for any $\epsilon>0$. 
Using the Chinese Remainder Theorem, we extend this to moduli $\modulus$ composed only of good primes; that is, if $\oddprime | \modulus$, then $\oddprime \in \goodprimes$. 
Let $\orderofprime{\modulus}$ denote the precise power of the prime $\oddprime$ dividing $\modulus$. 
If $\modulus$ is composed only of good primes, then 
%
\begin{equation}\label{eq:Euclidean_radii_good_prime_set_size}
\# \setof{i \in \Zmod{\modulus} : \radii_{i (\modulus)} \not= \emptyset} 
\leq \prod_{\oddprime | \modulus} \# \setof{\radius^2 \mod \oddprime^{\orderofprime{\modulus}} : \radius \in \radii} 
\lesssim_\epsilon \prod_{\oddprime | \modulus} \oddprime^{\orderofprime{\modulus}-1+\epsilon} 
\lesssim_{\epsilon} \modulus^{\epsilon} \cdot \totientfunction{\modulus}
\end{equation}
for any $\epsilon>0$. 

For a modulus $\modulus$, write $\modulus = \goodmodulus \cdot \badmodulus$ where $\goodmodulus$ is composed only of good primes while $\badmodulus$ is composed only of bad primes ($\badmodulus$ is composed of bad primes if $\oddprime | \badmodulus$ implies $\oddprime \in \badprimes$). 
In the case that a prime is both good and bad, we regard it as a bad prime in the following estimate. 
Now \eqref{eq:counting_bound_for_maximal_multipliers} of Proposition~\ref{prop:counting_bound_for_maximal_multipliers}, the Weil bound for Kloosterman multipliers \eqref{eq:Weil_bound_for_a_multiplier} and \eqref{eq:Euclidean_radii_good_prime_set_size} imply that, 
\begin{align*}
\ellpoperatornorm{\sup_{\radius \in \radii} \absolutevalueof{\HLop \fxn}} 
& 
\lesssim 
  \sum_{\modulus \in \N} \badmodulus^{1-\dimension/p'} 
  \cdot \# \setof{i \in \Zmod{\modulus} : \radii_{i (\modulus)} \not= \emptyset} 
  \cdot \sup_{i \in \Zmod{\modulus}} \inbraces{\ellpoperatornorm{U^{1, \modulus}_{\radius_i}}} 
\\ & 
\lesssim 
   \sum_{\modulus \in \N} \inparentheses{ \totientfunction{\goodmodulus}^{{2}/{p}} \cdot {\goodmodulus}^{\epsilon-\frac{\dimension-1}{p'}} \cdot \badmodulus^{1-\frac{\dimension}{p'}} } 
\\ & 
= 
   \prod_{\oddprime \in \goodprimes} \inparentheses{1 + \oddprime^{1-\frac{2}{p}+\epsilon} \sum_{\primepower \in \N} \oddprime^{\primepower (\frac{2}{p}-1-\frac{\dimension-1}{p'})}} 
      \cdot \prod_{\oddprime \in \badprimes} \inparentheses{1 + \sum_{k \in \N} \oddprime^{k(1-\frac{\dimension}{p'})}} 
\\ & 
\lesssim_p 
   \inbrackets{\prod_{\oddprime \in \goodprimes} \inparentheses{1 + \oddprime^{\epsilon-\frac{\dimension-1}{p'}}} }
   \cdot \inbrackets{\prod_{\oddprime \in \badprimes} \inparentheses{1 + \oddprime^{1-\frac{\dimension}{p'}}} } 
\\ & 
< 
   \zeta \inparentheses{\frac{\dimension-1}{p'}-\epsilon} \cdot 
  \prod_{\oddprime \in \badprimes} \inparentheses{1 + \oddprime^{1-\frac{\dimension}{p'}}} 
\end{align*}
for all sufficiently small $\epsilon>0$. 
The third inequality is true provided that $\frac{2}{p}-1-\frac{\dimension-1}{p'} < 0$ and $1-\frac{\dimension-1}{p'} < 0$; this is equivalent to $p > \frac{\dimension-1}{\dimension-2}$. 
The zeta--function converges for small enough $\epsilon>0$ if and only if $\frac{\dimension-1}{p'} > 1$. Unravelling this condition yields that we again require $p > \frac{\dimension-1}{\dimension-2}$. 
%
%
%
The second factor in the final inequality is bounded precisely when 
\(
\sum_{\oddprime \in \badprimes} \oddprime^{1-\frac{\dimension}{p'}} 
< 
\infty 
. 
\)
By assumption \eqref{bad_primes_estimate}, we assume that 
\(
\sum_{\oddprime \in \badprimes} \oddprime^{-s} 
< 
\infty 
\)
for some \( s \in (0,1] \). 
Taking \( 1-\dimension/p' \leq -s \), we require 
\( 
p 
\geq 
\frac{\dimension}{\dimension-(1+s)} 
. 
\)


\section{The error term}\label{section:error_term}
%
In this section we handle the error term. 
In particular we show that over an arbitrary lacunary subsequence, we can bound the error term on $\ell^p$ for $\frac{\dimension-1}{\dimension-2} < p \leq 2$. 
Before doing so, we prove a weaker bound that does not make use of cancellation in averages of Ramanujan sums, but is simpler, and suffices for Corollary~\ref{cor:Euclid}. 
%

\subsection{Preliminary bound for the error term}
In this section, we will bound the error term using the improved bound \eqref{eq:Kloosterman_error_bound} in the Kloosterman refinement of our operators and a simple bound on $\ell^1(\Z^\dimension)$ for the operators $C^\modulus_\radius$. 
\begin{lemma}\label{lemma:weak_error_term_bound}
Let $\dimension \geq 4$. 
Suppose that $\radii$ is a lacunary sequence. 
Then for $\frac{\dimension+1}{\dimension-1} < p \leq 2$, 
\begin{equation}\label{eq:weak_error_term_bound}
\lpnorm{p}{\sup_{\radius \in \radii} \absolutevalueof{\error \fxn}} 
\lesssim_p \lpnorm{p}{\fxn} 
. 
\end{equation}
\end{lemma}
The proof for $\ell^2$ is standard: bound the sup by a square function and apply the Kloosterman refinement of \eqref{eq:Kloosterman_error_bound}. To obtain our range of $p$, we will need a suitable bound on $\ell^1(\Z^\dimension)$. For this we have the following proposition. 
\begin{prop}\label{prop:trivial_ell1_bound_for_error_term}
For any modulus $\modulus \in \N$, 
\begin{equation}\label{eq:trivial_ell1_bound}
\lpnorm{1,\infty}{\Kloostermanoperator \fxn} 
\lesssim \frac{\radius \cdot \varphi(\modulus)}{\modulus} \lpnorm{1}{\fxn}
.
\end{equation}
\end{prop}
Here and throughout, $\varphi$ denotes Euler's totient function. 
With this bound, we can prove Lemma~\ref{lemma:weak_error_term_bound}. 
\begin{proof}[Proof of Lemma~\ref{lemma:weak_error_term_bound}]
For $\ell^2(\Z^\dimension)$ we have $\lpnorm{\infty}{\arithmeticFT{\error}(\xi)} \lesssim \radius^{-\delta}$ by \eqref{eq:Kloosterman_error_bound}
for all $\delta < \frac{\dimension-3}{2}$. 
On $\ell^1(\Z^\dimension)$ Proposition~\ref{prop:trivial_ell1_bound_for_error_term} implies that $\lpnorm{1,\infty}{C^\modulus_\radius \fxn} \lesssim \radius \lpnorm{1}{\fxn}$ so that 
\begin{equation}
\lpnorm{1, \infty}{\sum_{\modulus=1}^\radius C^\modulus_\radius \fxn} 
\lesssim \radius^2 \lpnorm{1}{\fxn}
\end{equation}
while 
$\lpnorm{1}{\avgop} \lesssim 1$ for each $\radius$ 
so that 
$\lpnorm{1,\infty}{\error \fxn} \lesssim \radius^2 \lpnorm{1}{\fxn}$. 
By interpolation, for $1 < p \leq 2$, 
\[
\lpnorm{p}{\error \fxn} 
\lesssim \radius^{2(\frac{2}{p}-1)} \cdot \radius^{-\delta(2-\frac{2}{p})} 
= \radius^{\frac{2}{p}(2+\delta)-2(1+\delta)} \lpnorm{p}{\fxn}
. 
\]
$\frac{2}{p}(2+\delta)-2(1+\delta) < 0$ if and only $p > \frac{2+\delta}{1+\delta}$. This holds for all  $\delta < \frac{\dimension-3}{2}$ which gives $p > \frac{\dimension+1}{\dimension-1}$. Sum over a lacunary set in this range of $p$ to obtain \eqref{eq:weak_error_term_bound}. 
\end{proof}
\begin{rem}
The best known bound for $\delta$ is all $\delta < \frac{\dimension-3}{2}$ by Magyar's version of Heath-Brown's Kloosterman refinement in \cite{Magyar_discrepancy}. 
Due to the existence of cusp forms, this is the best one can expect. 
\end{rem}
We are left to prove Proposition~\ref{prop:trivial_ell1_bound_for_error_term}. 
We use the structure of the kernel to prove a weak-type bound.
\begin{proof}[Proof of Proposition~\ref{prop:trivial_ell1_bound_for_error_term}]
For $t>0$, let $\dilateby{t}$ be the operator $\dilateby{t}\fxn(x) = \fxn(tx)$. 
Since 
\[
\latticeFT{C^\modulus_\radius}(\xi) 
= \sum_{\latticepoint \in \lattice} K(\modulus, \radius^2, \latticepoint) \dilateby{\modulus} \Psi(\xi - x/\modulus) \contFT{d\spheremeasure_\radius}(\modulus \xi - x ) 
, 
\]
one can calculate the kernel $\Kloostermankernel$ for the multiplier $\latticeFT{C^\modulus_\radius}$ and $x \in \R^\dimension$, 
\begin{equation}\label{eq:kernel_formula}
\Kloostermankernel(x) 
= \sum_{\unit \in \unitsmod{\modulus}} \eof{\frac{\unit (\radius^\degreetwo-\absolutevalueof{x}^{\degreetwo})}{\modulus}} \cdot \radius^{-\dimension} \contFT{\dilateby{\modulus/\radius} \Psi} \convolvedwith {d\spheremeasure}(x/\radius) 
. 
\end{equation}
A standard argument -- see \cite{Ionescu_spherical} -- shows 
\[
\contFT{\dilateby{\modulus/\radius} \Psi} \convolvedwith {d\spheremeasure}(x)
\lesssim (\modulus/\radius)^{-1} (1+\absolutevalueof{x})^{-2 \dimension}
.
\]
Then 
\[
\absolutevalueof{\Kloostermankernel(x)} 
\lesssim \radius^{1-\dimension} \inparentheses{1+\absolutevalueof{x/\radius}}^{-2 \dimension} 
. 
\]
$\radius^{-\dimension} \inparentheses{1+\absolutevalueof{x/\radius}}^{-2 \dimension}$ is an approximation to the identity which implies \eqref{eq:trivial_ell1_bound} by the Magyar--Stein--Wainger transference principle. 
\end{proof}

\subsection{The Ramanujan bound for the error term}
In this section we improve the bound \eqref{eq:trivial_ell1_bound} for the error term $\error$. 
The following lemma concludes the proof of Theorem~\ref{thm:variants_of_Euclid}. 
\begin{lemma}\label{lemma:error_term_lemma}
Let \(\dimension \geq 4\). 
If \(\radii\) forms a lacunary sequence
, then for \(\frac{\dimension-1}{\dimension-2} < p \leq 2\), 
\begin{equation}
\lpnorm{p}{\sup_{\radius \in \radii} \absolutevalueof{\error \fxn}} 
\lesssim_p \lpnorm{p}{\fxn} 
. 
\end{equation}
\end{lemma}
The strategy is the same as in Lemma~\ref{lemma:weak_error_term_bound} but we improve the bound on $\ell^1(\Z^\dimension)$ to the following. 
\begin{prop}\label{prop:improved_ell1_bound_for_maximal_operators}
For $\radius \in \acceptableradii$ and all \(\epsilon>0\), we have 
\begin{equation}\label{eq:improved_ell1_bound}
\lpnorm{1,\infty}{\sum_{\modulus = 1}^{\radius} \Kloostermanoperator \fxn} 
\lesssim_{\epsilon} \radius^{1+\epsilon} \lpnorm{1}{\fxn} 
.
\end{equation}
\end{prop}
%


%
%
%

%
The sums 
\begin{equation}
c_\modulus(N) 
:= \sum_{\unit \in \unitsmod{\modulus}} \eof{\frac{\unit N}{\modulus}}
\end{equation}
are known as \emph{Ramanujan sums} and clearly satisfy the bound $\absolutevalueof{c_\modulus(N)} \leq \phi(\modulus)$ for all $N$. However, there is an improved bound on average -- see (3.44) on page 126 of \cite{Bourgain_lattice_restriction_1}: 
\begin{equation}\label{eq:average_bound_for_Ramanujan_sums}
\sum_{Q \leq \modulus < 2Q} \absolutevalueof{c_\modulus(N)} 
= \sum_{Q \leq \modulus < 2Q} \absolutevalueof{\sum_{\unit \in \unitsmod{\modulus}} \eof{\frac{\unit N}{\modulus}}}
\lesssim Q \cdot d(N, Q)
\end{equation}
where $d(N, Q)$ is the number of divisors of $N$ up to $Q$. Therefore we can bound the above average of Ramanujan sums, \eqref{eq:average_bound_for_Ramanujan_sums} by $\lesssim_\epsilon Q \cdot N^{\epsilon}$ for all $\epsilon > 0$; that is, with a ``log-loss". 
Using the improved average bound for Ramanujan sums, we improve \eqref{eq:trivial_ell1_bound} to \eqref{eq:improved_ell1_bound}. 

\begin{proof}[Proof of Proposition~\ref{prop:improved_ell1_bound_for_maximal_operators}]
Again, for $t>0$, let $\dilateby{t}$ be the operator $\dilateby{t}\fxn(x) = \fxn(tx)$. 
We rewrite \eqref{eq:kernel_formula} as 
\begin{equation}
\Kloostermankernel(x) 
= c_\modulus(\radius^\degreetwo-\absolutevalueof{x}^{\degreetwo}) \cdot \radius^{-\dimension} \contFT{\dilateby{\modulus/\radius} \Psi} \convolvedwith {d\spheremeasure}(x/\radius)
\end{equation}
By the Magyar--Stein--Wainger transference principle, \eqref{eq:improved_ell1_bound} will follow from proving the pointwise bound for all $x \in \R^\dimension$ and any $\epsilon > 0$, 
\begin{equation}\label{eq:pointwise_bound_for_kernels}
\absolutevalueof{\sum_{\modulus = 1}^{\radius} \Kloostermankernel(x)}
\lesssim_{\epsilon} \radius^{1+\epsilon-\dimension} \inparentheses{1+\absolutevalueof{x/\radius}}^{-2 \dimension}
. 
\end{equation}
From 
\begin{equation}
\Kloostermankernel(x) 
= c_\modulus(\radius^\degreetwo-\absolutevalueof{x}^{\degreetwo}) \cdot \contFT{\dilateby{\modulus} \Psi} \convolvedwith {d\spheremeasure_\radius}(x) 
, 
\end{equation}
we easily see 
\begin{equation}
\sum_{\modulus = 1}^{\radius} \Kloostermankernel(x) 
= \inbrackets{\sum_{\modulus = 1}^{\radius} c_\modulus(\radius^\degreetwo-\absolutevalueof{x}^{\degreetwo}) \cdot \contFT{\dilateby{\modulus} \Psi}} \convolvedwith {d\spheremeasure_\radius}(x) 
. 
\end{equation}
Note that $\dilateby{\modulus} \Psi$ is supported in $[-1/4\modulus, 1/4\modulus]^\dimension$ for each $1 \leq \modulus \leq \radius$. 
Using an appropriate partition of unity for each $\dilateby{\modulus} \Psi$, we are able to sum over $q \leq \radius$ and use \eqref{eq:average_bound_for_Ramanujan_sums} to obtain \eqref{eq:pointwise_bound_for_kernels}. 
%
%
\end{proof}


\section{Proof of Theorem~\ref{thm:ENFANT}}\label{section:ENFANT_example} 

Fix $\sequencegrowth > 1$ a real number. 
In this section we prove Theorem~\ref{thm:ENFANT} for the collection of radii 
\[
\radii 
:= 
\setof{\radius_j \in \R^+ : \radius_j^2 = 1 + \prod_{i \leq h(j)} \oddprime_i}
\] 
where \(h(j) := 2^{j^\sequencegrowth}\). 
The proof for the remaining sequences in Theorem~\ref{thm:ENFANT} is similar, but notationally cumbersome. 

Let $\theprimes$ denote the set of primes in $\N$. 
We split the primes into bad primes and good primes as follows. 
Let the \emph{bad primes} $\badprimes$ be the set of primes dividing $\radius_j^2$ for some radius $\radius_j \in \radii$ together with the prime 2. 
Let the \emph{good primes} $\goodprimes := \theprimes \setminus \badprimes$ be the remaining primes. 
We enumerate the primes so that \( \oddprime_n \) denotes the \(n^{th}\) prime. 
The Prime Number Theorem says that \(\oddprime_n \sim n \log n\) as \(n \to \infty\).

If $\oddprime$ is a prime, then choose $J$ such that $\oddprime_{h(J)} \leq \oddprime < \oddprime_{h(J+1)}$. 
If $j > J$, then $\radius_j^2 \equiv 1 \mod \oddprime$, 
and 
\[
\# \setof{i \in \Zmod{\oddprime} : \radii_{i (\oddprime)} \not= \emptyset} 
\leq J+1 
\lesssim J^\sequencegrowth 
\lesssim \log \oddprime 
\]
where the last inequality follows from the Prime Number Theorem, with an implicit constant that is independent of the prime $\oddprime$. 
The last inequality is explained as follows. By the prime number theorem, $\oddprime_{2^{J^\sequencegrowth}} \eqsim 2^{J^\sequencegrowth} \cdot \log {2^{J^\sequencegrowth}} \eqsim 2^{J^\sequencegrowth} \cdot {J^\sequencegrowth}$. This implies that $\log \oddprime_{2^{J^\sequencegrowth}} \eqsim {J^\sequencegrowth} + \sequencegrowth \log J \eqsim J^\sequencegrowth$ so that $J < J^\sequencegrowth \lesssim \log \oddprime$. 
These estimates hold for every prime; in particular, they hold for the good primes.

An essential point is that there are few bad primes for our sequence; this is quantified by the following bound: 
\begin{equation}\label{eq:bad_primes_bound}
\sum_{\oddprime \in \badprimes} \oddprime^{-1} 
\leq 
\sum_{j=1}^\infty h(j) \cdot \oddprime_{h(j)}^{-1} 
\lesssim 
\sum_{j=1}^\infty h(j) \inbrackets{h(j) \log h(j)}^{-1} 
= 
\sum_{j=1}^\infty \inbrackets{\log h(j)}^{-1} 
. 
\end{equation}
The first inequality is true since each prime dividing $\radius_j$ is at least of size $\oddprime_{h(j)}$ and there are at most $h(j)$ prime divisors, and the last inequality follows from the Prime Number Theorem which says $\oddprime_n \eqsim n \log n$. 
Since $h(j) = 2^{j^\sequencegrowth}$ for some $\sequencegrowth > 1$, \eqref{eq:bad_primes_bound} converges. 
%

\section{Concluding remarks and open questions}\label{section:conclusion}
%
\begin{question}
Estimate \eqref{eq:bad_primes_bound} of the Dirichlet series 
\(
\sum_{\oddprime \in \badprimes} \oddprime^{-s} 
\) 
is rather crude. 
Improving this estimate would improve our range of $\ell^p(\Z^\dimension)$-spaces, potentially to $p > \frac{\dimension-1}{\dimension-2}$
. 
The author is unaware of any investigations of our Dirichlet series in the literature. 
\emph{
Does 
\(
\sum_{\oddprime \in \badprimes} \oddprime^{-s} 
\) 
converge for some \(s \in (0,1)\) for sequences related to Theorem~\ref{thm:ENFANT}? 
}
\end{question}

\begin{question}
\emph{
Can we prove \eqref{eq:bad_primes_bound} where \(h(j)\) grows more slowly such as \(h(j):=j\)? 
}
\end{question}

\begin{rem}
In section~\ref{subsection:Ackbar}, we mentioned that J. Zienkiewicz showed that Conjecture~\ref{conjecture:lacunary} fails in general. 
More generally, one can show that if \eqref{good_primes_estimate} is violated for infinitely many primes, then \(\sequentialarithmeticsphericalmaxfxn\) is unbounded on \(\ell^p(\Z^\dimension)\) for \(p\) close to 1 and 
\(\dimension \geq 5\). 
We revise Conjecture~\ref{conjecture:lacunary} to take into account this obstruction. 
\begin{conjecture}\label{conjecture:adelic_lacunary}
For $\dimension \geq 5$, if $\radii$ is a lacunary subsequence of $\acceptableradii$ such that \eqref{good_primes_estimate} holds for all but finitely primes \(\oddprime\), then $\sequentialarithmeticsphericalmaxop: \ell^1(\Z^\dimension) \to \ell^{1,\infty}(\Z^\dimension)$.
%
%
The same is true if \(\dimension = 4\) and 2 is a good prime. 
\end{conjecture}
\end{rem}

\begin{question}
There is an elegant characterization of the \(L^p(\R^\dimension)\)-boundedness of the continuous spherical maximal function over subsequences of \(\R^+\) in \cite{SWW_spherical}. 
\emph{Is there such a characterization for the discrete spherical averages?} 
Zienkiewicz's result shows that any such characterization must also account for arithmetic phenomena. 
\end{question}

\bigskip
\section*{Acknowledgements}
The author would like to thank Lillian Pierce for discussions on the arithmetic lacunary spherical maximal function and for pointing out a critical mistake in a previous version of this paper, and his advisor, Elias Stein, for introducing him to the problem. 
The author would also like to thank Roger Heath-Brown for a discussion on the limitations of the circle method, Peter Sarnak for discussions regarding Kloosterman sums and Jim Wright for explaining aspects of the continuous lacunary spherical maximal function. 
And a special thanks to Lutz Helfrich, James Maynard and Kaisa Matom\"aki at the Hausdorff Center for Mathematics's ENFANT and ELEFANT conferences in July 2014 for pointing out the family of sequences used in Theorem~\ref{thm:ENFANT}.

\bigskip
\bibliographystyle{amsalpha}
\bibliography{references_lacunary_arithmetic}
\end{document}